\documentclass{amsart}
\usepackage{amsmath,amsfonts,amssymb}
\usepackage{ifthen}
\usepackage{longtable}
\usepackage{array}
\usepackage{url}
\usepackage{multirow}
\usepackage{graphicx}
\usepackage{booktabs}
\usepackage{float}
\usepackage{tikz}

\newtheorem{example}{Example}
\newtheorem{definition}{Definition}
\newtheorem{theorem}{Theorem}
\newtheorem{lemma}{Lemma}
\newtheorem{corollary}{Corollary}

\newtheorem{claim}{Claim}

\theoremstyle{remark}
\newtheorem{remark}{Remark}

\begin{document}

\title[Generalized $LS$-sequences]{Discrepancy of generalized $LS$-sequences}


\author[M.R. Iac\`o]{Maria Rita Iac\`o}
\address{M.R. Iac\`o \newline
\indent Graz University of Technology, \newline
\indent Institute of Mathematics A,\newline
\indent  Steyrergasse 30, 8010 Graz, Austria.}
\email{iaco\char'100math.tugraz.at}

\author[V. Ziegler]{Volker Ziegler}
\address{V. Ziegler \newline
\indent Institute of mathematics,\newline
\indent Universtity of Salzburg,  \newline
\indent Hellbrunner Strasse 34, \newline
\indent 5020 Salzburg, Austria.}
\email{volker.ziegler\char'100sbg.ac.at}

\thanks{The first author was supported by the Austrian Science Fund (FWF) Project F5510 (part of the Special Research
Program (SFB) \textquotedblleft Quasi-Monte Carlo Methods: Theory and Applications\textquotedblright) and partially supported by the Austrian
Science Fund (FWF): W1230, Doctoral Program ``Discrete Mathematics''. The second author was supported by the Austrian Science Fund (FWF) under the project 
P~24801-N26.}

\begin{abstract}
The $LS$-sequences are a parametric family of sequences of points in the unit interval. They were introduced by Carbone \cite{carbone}, who also proved that 
under an appropriate choice of the parameters $L$ and $S$, such sequences are low-discrepancy.
The aim of the present paper is to provide explicit constants in the bounds of the discrepancy of $LS$-sequences. Further, we generalize the construction 
of Carbone \cite{carbone} and construct a new class of sequences of points in the unit interval, the generalized $LS$-sequences.
\end{abstract}


\maketitle

\section{Introduction}
A sequence $(x_n)_{n \in \mathbb{N}}$ of points in $[0,1)$ is called
\textit{uniformly distributed modulo 1} (u.d. mod 1) if
\begin{equation*}
 \lim_{N \rightarrow \infty} \frac{1}{N}\sum_{n = 1}^N \mathbf{1}_{[a, b)} (x_n)
= \lambda([a,b))
\end{equation*}
for all intervals $[a, b ) \subseteq [0,1)$. A further
characterization of u.d.\ is given by the following well-known result of
Weyl~\cite{weyl}: a sequence $(x_n)_{n \in \mathbb{N}}$ of points in $[0,1)$
is u.d. mod 1 if and only if for every continuous function $f$ on $[0,1)$ the
relation
\begin{equation*}
  \lim_{N \rightarrow \infty} \frac{1}{N}\sum_{n = 1}^N f(x_n) = \int_{[0,1)}
f(x) dx
\end{equation*}
holds. 

An important quantity called \emph{discrepancy} is introduced when dealing with u.d. sequences. It measures the maximal deviation between the empirical 
distribution of a sequence and the uniform distribution.
Let $\omega_N = \{x_1,\dots,x_N\}$ be a finite set of real numbers in $[0,1]$. The quantity
\begin{equation*}
D_N(\omega_N)=\sup_{0\leq a< b\leq 1} \left|\frac{1}{N} \sum_{n=1}^N \mathbf{1}_{[a,b)}(x_n)-(b-a)\right|
\end{equation*}
is called the discrepancy of the given set $\omega_N$.
This definition naturally extends to infinite sequences $(x_n)_{n\in\mathbb{N}}$ by associating to it the sequence of positive real numbers $D_N (\{x_1, x_2, 
\dots, x_N \})$. We denote by $D_N(x_n)$ the discrepancy of the initial segment $\{x_1, x_2, \dots, x_N\}$ of the infinite sequence.

Sometimes it is also useful to restrict the family of intervals considered in the definition of discrepancy to intervals of the form $[0, a)$ 
with $0 < a \leq 1$. This leads to the following definition of star-discrepancy
\begin{equation*}
D_N^*(\omega_N)=\sup_{0< a\leq 1} \left|\frac{1}{N} \sum_{n=1}^N \mathbf{1}_{[0,a)}(x_n)-a\right|\ .
\end{equation*}
It is a well known result that a sequence $(x_n)_{n\in\mathbb{N}}$ of points in $[0, 1]$ is u.d. if and only if 
\begin{equation*}
\lim_{N\to\infty} D_N^*(x_n)=0\ .
\end{equation*}

Sequences whose discrepancy is of order $\mathcal{O}(N^{-1}\log N)$ are called \emph{low-discrep\-ancy} sequences.
These sequences are of particular interest in the theory of numerical integration and are used in the \emph{Quasi-Monte Carlo (QMC)} integration.
For more information on discrepancy theory, low-discrepancy sequences and QMC integration see~\cite{dt, kn}.

We now recall the splitting procedure that gives rise to the generalized $LS$-sequences as sequences of points associated to a sequence of partitions of the 
unit interval.

\begin{definition}[Kakutani splitting procedure]
If $\alpha \in (0,1)$ and $\pi = \{[t_{i-1}, t_i): 1 \leq i \leq k\}$ is any
partition of $[0,1)$, then $\alpha \pi$ denotes its so-called
$\alpha$-refinement, which is obtained by subdividing all intervals of $\pi$
having maximal length into two parts, proportional to $\alpha$ and $1- \alpha$,
respectively. The so-called Kakutani's sequence of partitions $(\alpha^n
\omega)_{n \in \mathbb{N}}$ is obtained as the successive $\alpha$-refinement of
the trivial partition $\omega = \{[0,1)\}$.
\end{definition}

The notion of $\alpha$-refinements can be generalized in a natural way to
so-called $\rho$-refinements.

\begin{definition}[$\rho$-refinement]
Let $\rho$ denote a non-trivial finite partition of $[0,1)$. Then the
$\rho$-refinement of a partition $\pi$ of $[0,1)$, denoted by $\rho \pi$, is
given by subdividing all intervals of maximal length positively homothetically
to $\rho$. Note that the $\alpha$-refinement is a special case with $\rho =
\{[0, \alpha), [\alpha , 1)\}$.
\end{definition}

By a classical result due to Kakutani~\cite{kakutani}, for any $\alpha$ the sequence
of partitions  $(\alpha^n \omega)_{n \in \mathbb{N}}$ is uniformly distributed,
which means that for every interval $[a,b] \subset [0,1]$,
\begin{equation*}
 \lim_{n \rightarrow \infty} \frac{1}{k(n)}\sum_{i=1}^{k(n)}
\mathbf{1}_{[a,b]}(t_i^n) = b-a,
\end{equation*}
where $k(n)$ denotes the number of intervals in $\alpha^n
\omega=\{[t_{i-1}^n,t_i^n),~1 \leq i \leq k(n)\}$. The same result holds for any
sequence of $\rho$-refinements of $\omega$, due to a result of
Vol\v{c}i\v{c}~\cite{volcic} (see also~\cite{ah,drmota}).

The generalized $LS$-sequence of partitions represent a special case of a $\rho$-refine\-ment. 

\begin{definition}[Generalized $LS$-sequence of partitions]\label{Def:LS_partition}
Let $L_1,\dots,L_k$ be non-negative integers, with $L_1L_k\neq 0$.
We define the generalized $LS$-sequence of partitions $(\rho^n_{L_1,\dots, L_k} \omega)_{n \in
\mathbb{N}}$ as the successive $\rho$-refinement of the trivial partition
$\omega$, where $\rho_{L_1,\dots, L_k}$ consists of $L_1+L_2+\dots +L_k$ intervals such that $L_1$
has length $\beta$, $L_2$ has length $\beta^2$ and so on up to $L_k$ having length $\beta^k$.
\end{definition}

Note that necessarily $L_1\beta+\dots +L_k\beta^k = 1$
holds, and consequently for each $k$-tuple $(L_1,\dots,L_k)$ of parameters satisfying the assumptions made in Definition \ref{Def:LS_partition}
there exists exactly one positive real number $\beta$ satisfying $L_1\beta+\dots+L_k\beta^k=1$. 

If $k=2$, then we obtain the definition of the classical $LS$-sequence of partitions introduced by Carbone \cite{carbone}.
Given a sequence of partitions we can 
assign a sequence of points by ordering the left endpoints of the intervals in the partition. The corresponding $LS$-sequences of points, denoted by 
$(\xi_{L,S}^n)_{n\in\mathbb{N}}$, have been introduced by Carbone \cite{carbone}, who proved that 
whenever $L\geq S$ there exists a positive constant $k_1$ such that
\begin{equation}\label{eq:Discrepancy_LS_general}
D_N(\xi_{L,S}^{1}, \xi_{L,S}^{2},\ldots, \xi_{L,S}^{N})\leq k_1\frac{\log N}{N}\ .
\end{equation}
One purpose of the present paper is to give an estimate for $k_1$. This problem still open in \cite{carbone} can be solved by the numeration approach 
presented by Aistleitner et.al.~\cite{ahz}. However the main focus of the present paper is to generalize the construction of classical $LS$-sequences of points by 
Carbone~\cite{carbone}. Therefore we introduce a numeration system for the integers which is the generalization of the numeration system introduced by 
Aistleitner et.al.~\cite{ahz}. This numeration system will lead us to the definition of generalized $LS$-sequences of points. This definition by a numeration 
system allows us to prove several auxiliary results on the distribution of generalized $LS$-sequences which were proved in the case that $k=2$ by 
Carbone~\cite{carbone} and Aistleitner et.al.~\cite{ahz}. All this is the content of the next 
section. In Section \ref{Sec:ProofTh1} we explicitly compute the constant $k_1$ in \eqref{eq:Discrepancy_LS_general} and obtain explicit upper bounds for 
the discrepancy of the classical $LS$-sequences (see Theorem~\ref{constant}). The computation of the discrepancy in the classical case, i.e. the case that
$k=2$ will pave the way 
to compute explicit bounds for the discrepancy of generalized $LS$-sequences, which we compute in the final section of this paper.

\section{Generalized $LS$-sequences}

Let us consider the generalized $LS$-sequence of partitions. Then the partition $\rho^n_{L_1,\dots, L_k} \omega$
consists of intervals having lengths $\beta^n,\dots ,\beta^{n+k-1}$ and this
fact makes the analysis of the generalized $LS$-sequences more complicated, compared to the
analysis of the classical ones, where only two lengths are considered. We denote by $t_n$ the total number of
intervals of $\rho^n_{L_1,\dots, L_k} \omega$, and correspondingly by $l_{n,1}, \dots , l_{n,k}$ the number of intervals of the $n$-th partition having length 
$\beta^n, \dots ,\beta^{n+k-1}$ respectively.

From a general point of view the only canonical restrictions to the $k$-tuple $(L_1,\dots,L_k)$ are that the $L_i$ are non-negative integers for all 
$i=1,\dots,k$ such that $L_1L_k\neq 0$. However, we are interested in low-discrepancy sequences and we will see (Remark \ref{rem:why_Pisot})
that we obtain low 
discrepancy sequences if and only if all roots but one (counted with multiplicity) of the polynomial $L_kX^k+\dots+L_1 X-1$
have absolute value smaller than one.
Therefore we will assume from now on that the polynomial $L_kX^k+\dots+L_1 X-1$ has no double zeros and that 
there is exactly one root $0<\beta<1$ and all other roots have absolute value less than $1$. Let us remark that excluding multiple zeros is
mainly to avoid technical issues and similar results would be obtained if we only assume that the unique root $0<\beta<1$ is simple.

Furthermore let us note that, if we assume that
\begin{equation}
\label{eq:coeff_restriction} L_1\geq L_2\geq \dots\geq L_k>0,
\end{equation}
then $L_kX^k+\dots+L_1 X-1$ is irreducible and is the minimal polynomial of the reciprocal of a Pisot-number (see \cite{Brauer}), i.e. a $k$-tuple $(L_1,\dots,L_k)$
which satisfies \eqref{eq:coeff_restriction} also meets our assumptions made above. 

Let us write $\beta_1=\beta$ and let $\beta_2,\dots,\beta_k$ be the other roots 
of $L_kX^k+\dots+L_1 X-1$. When the coefficients of the polynomial fullfill condition \eqref{eq:coeff_restriction}, then the $\beta_2,\dots,\beta_k$ are the Galois 
conjugates of $\beta$ in some order. Note that since $1/\beta$ is a Pisot-number by our assumptions we have that $|\beta_i|>1$ for $i=2,\dots,k$.

In a first step we consider the quantities $t_n$ and $l_{n,i}$ for $n\geq 0$ and $i=1,\dots, k$. First we observe that
in the $n$-th partition step the longest $l_{n-1,1}$ intervals are divided into $L_1+\dots+L_k$ intervals, where $L_1$ intervals have length
$\beta^n$, $L_2$ intervals have length $\beta^{n+1}$ and so on. Hence we have that
\begin{equation}\label{eq:Rec}
\begin{split}
t_{n}=&\ l_{n-1,1}(L_1+\dots+L_k)+l_{n-1,2}+\dots+l_{n-1,k}\\
=&\ t_{n-1}+l_{n-1,1}(L_1+\dots+L_k-1),\\
l_{n,1}=&\ l_{n-1,2}+L_1l_{n-1,1},\\
&\vdots\\
l_{n,k-1}=&\ l_{n-1,k}+L_{k-1}l_{n-1,1},\\
l_{n,k}=&\ L_{k}l_{n-1,1}
\end{split}
\end{equation}
for all $n>0$. Of course the interval $[0,1)$ yields the initial conditions $t_0=l_{0,1}=1$ and $l_{0,2}=\dots=l_{0,k}=0$ and therefore
we can recursively compute the quantities $t_n,l_{n,1},\dots,l_{n,k}$ for all $n>0$. Since from the recursion point of view the
sequence $(t_n)_{n\geq 0}$ is closely 
related to the sequences $(l_{n,i})_{n\geq 0}$, we define $l_{n,0}=t_n$ in order to state several of our results in a compact way.
Moreover we put $l_{n,j}=0$ if $j>k$ or $n<0$. For our purposes we desire an explicit formula for the quantities $l_{n,i}$:

\begin{lemma}\label{lem:Rec}
 The sequences $(l_{n,i})_{n\geq 0}$, with $i=0,\dots, k$ satisfy the recursion
 $$l_{n,i}=L_1 l_{n-1,i}+\dots+L_kl_{n-k,i}, \qquad n\geq k.$$
 
 In particular, there exist explicit computable constants 
$\lambda_{j,i}$ for $0\leq i \leq k$ and $1\leq j \leq k$ such that
\begin{equation}\label{eq:Rec_explicit}
l_{n,i}=\sum_{j=1}^k \lambda_{j,i} \beta_j^{-n}
\end{equation}
\end{lemma}

\begin{proof}
First, we observe that by inserting the last line of \eqref{eq:Rec} into the second to last line we obtain
$$l_{n,k-1}=L_k l_{n-2,1}+L_{k-1}l_{n-1,1}$$
Now inserting this identity into the third to last line of \eqref{eq:Rec} and going on we end up with 
$$l_{n,1}=L_1 l_{n-1,1}+\dots+L_kl_{n-k,1},$$
hence we proved the first part of the lemma for $i=1$. Since the characteristic polynomial of this recursion is
$$X^k-L_1 X^{k-1}-\dots-L_k$$
there exist constants $\lambda_{j,1}$, with $j=1,\dots,k$ such that
$$l_{n,1}=\sum_{j=1}^k \lambda_{j,1} \beta_j^{-n}.$$
Note that the $1/\beta_j$ for $j=1,\dots, k$ are the roots of the characteristic polynomial.

Since $l_{n,i}=l_{n+1,1}-L_il_{n,1}$ we obtain \eqref{eq:Rec_explicit}, with
$\lambda_{j,i}=(\beta_j^{-1}-L_i)\lambda_{j,1}$ for $i=2,\dots,k$ and $j=1,\dots,k$. And since $t_n=l_{n,0}=l_{n,1}+\dots+l_{n,k}$ 
we see that \eqref{eq:Rec_explicit} also holds for $i=0$, with $\lambda_{j,0}=\lambda_{j,1}+\dots+\lambda_{j,k}$ for $j=1,\dots,k$.
On the other hand every sequence of the form \eqref{eq:Rec_explicit} 
fulfills a recursion of the required form and the proof of the lemma is complete. 
\end{proof}

Let us note that the explicit computation of the $\lambda$'s is easy for a given $k$-tuple $(L_1,\dots,L_k)$.
Indeed one can compute the values of
$l_{n,i}$ for all $i=0,1,\dots,k$ and $n=0,1,\dots,k-1$ by using the recursion \eqref{eq:Rec}. Therefore \eqref{eq:Rec_explicit} gives for each $i=0,1,\dots,k$ 
a linear inhomogeneous system with unknowns $\lambda_{j,i}$. Solving for the $\lambda$'s by Cramer's rule we obtain
\begin{equation}\label{eq:lambdas}
\lambda_{j,i}=\frac{\left| \begin{array}{ccccccc}
\beta_1^0 & \dots& \beta_{j-1}^0 & l_{0,i} & \beta_{j+1}^0  &\dots & \beta_k^0 \\
\beta_1^{-1} & \dots& \beta_{j-1}^{-1} & l_{1,i} & \beta_{j+1}^{-1}  &\dots & \beta_k^{-1} \\
\vdots & \ddots&\vdots &\vdots & \vdots&\ddots & \vdots \\
\beta_1^{1-k} & \dots& \beta_{j-1}^{1-k} & l_{k-1,i} & \beta_{j+1}^{1-k}  &\dots & \beta_k^{1-k}
\end{array} \right|}
{\left| \begin{array}{ccccccc}
\beta_1^0 & \dots& \beta_{j-1}^0 & \beta_j^0 & \beta_{j+1}^0  &\dots & \beta_k^0 \\
\beta_1^{-1} & \dots& \beta_{j-1}^{-1} & \beta_j^{-1} & \beta_{j+1}^{-1}  &\dots & \beta_k^{-1} \\
\vdots & \ddots&\vdots &\vdots & \vdots&\ddots & \vdots \\
\beta_1^{1-k} & \dots& \beta_{j-1}^{1-k} & \beta_j^{1-k} & \beta_{j+1}^{-k}  &\dots & \beta_k^{1-k}
\end{array} \right|}.
\end{equation}
 
Let us also state another property of the sequences $(l_{n,i})_{n\geq 0}$ for $i=1,\dots,k$, which we need at several places
in the construction of generalized $LS$-sequences:

\begin{lemma}\label{lem:sum_growth}
 We have $l_{n,1}+\dots+l_{n,m} \geq l_{n-1,1}+\dots+l_{n-1,m+1}$ for all $m\geq 1$.
\end{lemma}

\begin{proof}
  If $m\geq k$ this is clear since the sequences $l_{n,m}$ are 
all strictly monotone increasing with $n$ for $m=1,\dots,k$ and $l_{n,m}=0$ if $m>k$. In case that $m=1,\dots,k-1$ we have
\begin{equation*}
\begin{split}
l_{i-1,1}+\dots&+l_{i-1,m+1}\\
= &\, l_{i-1,1}+(l_{i,1}-L_1l_{i-1,1})+\dots+(l_{i,m}-L_ml_{i-1,1})\\
\leq &\, l_{i,1}+\dots+l_{i,m}.
\end{split}
\end{equation*}
\end{proof}

Our next step is to introduce the numeration system which will be the basis for our construction
of the generalized $LS$-sequence. Let $N\geq 0$ be a fixed integer and choose $n$ such that $t_n\leq
N < t_{n+1}$. We construct finite sequences $(\epsilon_m)_{0 \leq m \leq n}$,
$(\eta_m)_{0 \leq m \leq n}$, $(N_m)_{0 \leq m \leq n}$ and $(T_m)_{0 \leq m \leq n}$ recursively in the following way: 

\begin{itemize}
\item First we put $N_n=N$, $T_n=t_n$, $\epsilon_n=1$ and $\eta_n=\lfloor (N-T_n)/l_{n,1}\rfloor$.
\end{itemize}

Assume that we have computed the quantities
$N_{i},T_{i},\epsilon_{i}$ and $\eta_{i}$ for all indices $n\geq i >m\geq 0$.

\begin{itemize}
\item We denote by $j$ the unique integer
such that $\epsilon_{m+1}=\dots=\epsilon_{m+j}=0$ and $\epsilon_{m+j+1}\neq 0$,
in particular if $\epsilon_{m+1}\neq 0$ we put $j=0$.
\item Now we compute $N_{m}$ and $T_{m}$: 
\begin{equation*}
 \begin{split}
N_m=&\ N_{m+1}-\epsilon_{m+1}T_{m+1}-\eta_{m+1}l_{m+1,1}\\
T_m=&\ l_{m,1}+\dots+l_{m,j+2}.
\end{split}
\end{equation*}
\item If $N_m<T_m$ we put $\epsilon_m=\eta_m=0$. Otherwise we put $\epsilon_m=1$ and
$$\eta_m=\lfloor (N_m-T_m)/l_{m,1}\rfloor.$$
\end{itemize}



With this definition we obtain a representation for $N$ of the form
\begin{equation}\label{Rep:N}
 N=\sum_{i=0}^{n} (\epsilon_i T_i+\eta_il_{i,1}).
\end{equation}

Obviously $\epsilon_i\in\{0,1\}$ for all $i=0,1,\dots,n$. Let us note that we also have $0\leq \eta_i\leq L_1+\dots+L_k-2$
for all $i=0,1,\dots,n$. Indeed assume to the contrary that $\eta_m\geq L_1+\dots+L_k-1$ and that
$\epsilon_{m+1}=\dots=\epsilon_{m+j}=0$ but $\epsilon_{m+j+1}\neq 0$.
Then we would obtain in case that $j>0$
\begin{equation*}
\begin{split}
T_{m+1}>&\,  N_{m+1}\geq T_m+(L_1+\dots+L_k-1)l_{m,1}\\
= &\, l_{m,1}+l_{m,2}+\dots+l_{m+j+2}+(L_1+\dots+L_k-1)l_{m,1}\\
\geq &\, (L_1 l_{m,1}+l_{m,2})+\dots+(L_{j+1} l_{m,1}+l_{m,j+2})\\
= &\, l_{m+1,1}+\dots+l_{m+1,j+1}=T_{m+1}
\end{split}
\end{equation*}
a contradiction. In case that $j=0$ we similarly have
\begin{equation*}
\begin{split}
T_{m+1}\geq &\, l_{m+1,1}+l_{m+1,2}>l_{m+1,1}\\
\geq &\, N_m \geq T_m+(L_1+\dots+L_k-1)l_{m,1}\\
= &\, l_{m,1}+l_{m,2}+\dots+l_{m+j+2}+(L_1+\dots+L_k-1)l_{m,1}\\
\geq &\, (L_1 l_{m,1}+l_{m,2})+\dots+(L_{j+1} l_{m,1}+l_{m,j+2})\\
= &\, l_{m+1,1}+\dots+l_{m+1,j+1}=T_{m+1}
\end{split}
\end{equation*}
again a contradiction.

The following lemma gives a bijection between the integers and digit-expansions of the form given above. Note that the following lemma is
a generalization of a result due to Aistleitner et.al. for the classical $LS$-sequences \cite[Lemma 3]{ahz}.

\begin{lemma}\label{Lem:Digits}
There is a bijection between positive integers and finite sequences of the form
\[\mathcal D=((\epsilon_n,\eta_n),\ldots,(\epsilon_0,\eta_0))\]
such that $\epsilon_i\in\{0,1\}$, $\epsilon_n=1$, $0\leq \eta_i\leq L_1+\dots +L_k-2$ for all $0\leq i\leq n$,
$\epsilon_i=0$ implies $\eta_i=0$ and for all $1\leq m\leq k-1$ we have that $\eta_i\geq L_1+\dots+L_m-1$ implies
$\epsilon_{i+m}=0$.

This bijection is given by
\[\Psi(\mathcal D)=\sum_{i=0}^n (\epsilon_iT_i+\eta_il_{i,1})\]
and its inverse
\[\Phi(N)=((\epsilon_n,\eta_n),\ldots,(\epsilon_0,\eta_0)),\]
where the $T_i$, $\epsilon_i$ and $\eta_i$ are computed by the algorithm described
above.
\end{lemma}

\begin{proof}
In order to prove bijectivity we have to show that for every integer $N$ and every finite sequence $\mathcal D$ we have
$\Psi(\Phi(N))=N$ and $\Phi(\Psi(\mathcal D))=\mathcal D$. The first equation is evident from the presented algorithm, i.e. $\Phi$ is
injective. Thus
we are left to prove that $\Phi(\Psi(\mathcal D))=\mathcal D$, i.e. $\Phi$ is surjective. The proof is technical
and we proceed in several steps:\\

\noindent\textbf{Step I:} We show that 
$$\Phi(N)=((\epsilon_n,\eta_n),\ldots,(\epsilon_0,\eta_0))$$
such that the $\epsilon_i$ and $\eta_i$ for all $i\geq 0$ satisfy the conditions of the lemma. Thus we prove that $\Phi(N)$ is well defined.

Note that from the algorithm it is evident that 
$\epsilon_i\in\{0,1\}$, $\epsilon_n=1$ and that $\epsilon_i=0$ implies $\eta_i=0$.
Moreover, in the discussion after \eqref{Rep:N} we have shown that $0\leq 
\eta_i\leq L_1+\dots +L_k-2$ for all $0\leq i\leq n$. Therefore we have to show that $\eta_i\geq L_1+\dots+L_m-1$ implies
$\epsilon_{i+m}=0$.

We proceed by induction on $m$.
We start with the induction basis $m=1$, i.e. we assume that $\eta_i\geq L_1-1$. Let us assume for the moment that $\epsilon_{i+1}\neq 0$.
Then we get that
$$N_i\geq \epsilon_iT_i+\eta_il_{i,1}\geq l_{i,1}+l_{i,2}+(L_1-1)l_{i,1}=l_{i+1,1}>N_i$$
a contradiction and therefore we conclude that $\epsilon_{i+1}=0$, i.e. we have proved the induction basis.

Now, let us assume that $\eta_i\geq L_1+\dots+L_{m-1}-1$ implies $\epsilon_{i+m-1}=0$ for all $1\leq m \leq M-1$ and assume that
$\eta_i\geq L_1+\dots+L_{M}-1$. We aim to show that $\epsilon_{i+M}\neq 0$ yields a contradiction.
By the induction basis, see the paragraph above, we may assume that $M\geq 2$. Moreover by induction we may assume that $\epsilon_{i+m}=0$ for
all $1\leq m <M$.
Assuming $\epsilon_{i+M}\neq 0$
implies $T_{i}= l_{i,1}+\dots+l_{i,M+1}$ and $T_{i+1}=l_{i+1,1}+\dots+l_{i+1,M}$. Therefore we deduce 
\begin{align*}
T_{i+1}>& N_{i+1}=N_i\geq T_i+\eta_i l_{i,1}\\
\geq &\, l_{i,1}+\dots+l_{i,M+1}+(L_1+\dots+L_{M}-1)l_{i,1}\\
=& (L_1 l_{i,1}+l_{i,2})+\dots+(L_{M} l_{i,1}+l_{i,M+1})\\
=&\, l_{i+1,1}+\dots+l_{i+1,M}=T_{i+1}
\end{align*}
a contradiction, i.e. $\epsilon_{i+M}= 0$.\\

\noindent\textbf{Step II:}
It is enough to show that there are exactly $t_n-1$ sequences of length $\leq n$ satisfying the conditions of the Lemma. 

Indeed, we have already seen that $\Phi$ is injective. Therefore we have to show that
$\Phi$ is surjective. In particular, it is enough to prove that $\Phi$ induces a surjective map between the positive integers $<t_n$ 
and sequences of length $\leq n$, which satisfy the restrictions of the Lemma.
Therefore we have to prove the following claim:

\begin{claim}
There are exactly $t_n-1$ sequences of length $\leq n$ satisfying the conditions of the Lemma. Moreover,
there are exactly $l_{n-1,1}$ such sequences of length $n$ of the form $\mathcal D=((1,0),\dots)$.
\end{claim}

We will prove that claim be induction.\\

\noindent\textbf{Step III:} The induction basis is evident, since there are exactly $t_1-1=L_1+\dots+L_k-1$ sequences of length $1$. Moreover
there is only one sequence of length $1$ starting with the pair $(1,0)$.\\

\noindent\textbf{Step IV:} Suppose the claim is true for all integers $M<n+1$. We show that it is also true for $M=n+1$. 

First, let us show that there are exactly $l_{n,1}$ sequences of length $n+1$ starting with $(1,0)$. Let $1\leq m \leq k$, then 
there are exactly $(L_1+\dots+L_m -1)l_{n-m,1}$ sequences of the form
$$((1,0),\stackrel{m-1\;\; \text{times}}{\overbrace{(0,0),\dots,(0,0)}},(1,\eta),\dots)$$
 by induction hypothesis and the assumptions of the lemma
Similarly there are exactly $t_{n-k}$ sequences of the form
$$((1,0),\stackrel{\text{at least $k$ times}}{\overbrace{(0,0),\dots,(0,0)}},(1,\eta),\dots).$$
Therefore the number of sequences of length $n+1$ starting with $(1,0)$ is
\begin{equation*}
\begin{split}
 (L_1-1)l_{n-1,1}&+(L_1+L_2-1)l_{n-2,1}+\dots\\+&(L_1+\dots+L_k-1)l_{n-k,1}+t_{n-k}=\\
 (L_1-1)l_{n-1,1}&+(L_1+L_2-1)l_{n-2,1}+\dots\\+&(L_1+\dots+L_{k-1}-1)l_{n-k+1,1}+t_{n-k+1}=\\
 (L_1-1)l_{n-1,1}&+(L_1+L_2-1)l_{n-2,1}+\dots\\+&(L_1+\dots+L_{k-1}-1)l_{n-k+1,1}+l_{n-k+1,1}+\dots+l_{n-k+1,k}.
 \end{split}
\end{equation*}
Now note that
\begin{equation*}
\begin{split}
(L_1+&\dots+L_{m-1}-1)l_{n,1}+l_{n,1}+\dots+l_{n,m}\\
&=\,(L_1 l_{n,1}+l_{n,2})+\dots+ (L_{m-1} l_{n,1}+l_{n,m})\\
&=\,l_{n+1,1}+\dots+l_{n+1,m-1}
\end{split}
\end{equation*}
applying this identiy to the equation above we obtain that there are
\begin{equation*}
\begin{split}
 (L_1-1)l_{n-1,1}&+(L_1+L_2-1)l_{n-2,1}+\dots\\+&(L_1+\dots+L_{k-1}-1)l_{n-k+1,1}+l_{n-k+1,1}+\dots+l_{n-k+1,k}=\\
 (L_1-1)l_{n-1,1}&+(L_1+L_2-1)l_{n-2,1}+\dots\\+&(L_1+\dots+L_{k-2}-1)l_{n-k+1,1}+l_{n-k+2,1}+\dots+l_{n-k+2,k-1}=\\
 &\vdots\\
 (L_1-1)l_{n-1,1}&+l_{n-1,1}+l_{n-1,2}=L_1l_{n-1,1}+l_{n-1,2}=l_{n,1}
 \end{split}
\end{equation*}
sequences of length $n+1$ starting with $(0,1)$.
Therefore there are $(L_1+\dots+L_k-1)l_{n,1}$ sequences of lenght $n+1$ starting with $(1,\eta)$ and $0\leq \eta\leq L_1+\dots+L_k-2$ and
by induction there are $t_n-1$ sequences of length $<n+1$. Therefore all togehter there are 
$$(L_1+\dots+L_k-1)l_{n,1}+t_n-1=t_{n+1}-1$$
sequences of length $\leq n+1$, satisfying the conditions of the Lemma.
\end{proof}

The numeration system \eqref{Rep:N} allows us to generate the generalized $LS$-sequences of points in a direct way, as clarified in the following Definition

\begin{definition}\label{seq_of_point}
Let $N$ be an integer with representation given in \eqref{Rep:N}. Then
\begin{align*}
\xi_{L_1,\dots, L_k}^N = & \sum_{i=0}^n \left(\beta^{i+1}\min\{L_1,\epsilon_i+\eta_i\}\right.\\
&+\beta^{i+2}(\max\{\epsilon_i+\eta_i-L_1,0\}-\max\{\epsilon_i+\eta_i-L_1-L_2,0\})\\
&+\beta^{i+3}(\max\{\epsilon_i+\eta_i-L_1-L_2,0\}-\max\{\epsilon_i+\eta_i-L_1-L_2-L_3,0\})+\\
& \vdots\\
&+\beta^{i+k-1}(\max\{\epsilon_i+\eta_i-L_1-\dots-L_{k-2},0\}\\
&\quad-\max\{\epsilon_i+\eta_i-L_1-\dots -L_{k-1},0\})\\
&+\left.\beta^{i+k}(\max\{\epsilon_i+\eta_i-L_1-\dots-L_{k-1},0\})\right)\ .
\end{align*}
\end{definition}

The easiest example of a generalized $LS$-sequence is obtained by considering only three possible lengths for the intervals determining the partition. We call 
such sequences $LMS$-sequences.

\begin{example}\label{ex:LMS}
Consider the $LMS$-sequence of partitions $(\rho^n_{L,M,S} \omega)_{n\in\mathbb{N}}$ defined by the equation $\beta^3+\beta^2+2\beta-1=0$. 
In the following we describe the sequence of partitions and the associated sequence of points obtained by the procedure introduced above.
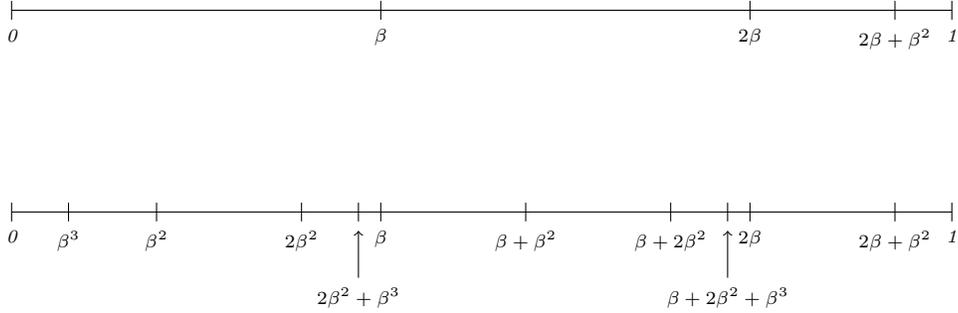
\begin{figure}[h!]
\begin{center}
\begin{tikzpicture}[scale=12.5]
\draw (0, 0) -- (1,0);
\draw (0,-0.01) node[below, black]{\scriptsize 0} -- (0,0.01);
\draw (1,-0.01) node[below, black]{\scriptsize 1} -- (1,0.01);
\draw (0.392647,-0.01) node[below, black]{\scriptsize$\beta$}-- (0.392647,0.01);
\draw (0.939465,-0.01) node[below, black]{\scriptsize$2\beta+\beta^2$}-- (0.939465,0.01);
\draw (0.785294,-0.01) node[below, black]{\scriptsize$2\beta$}-- (0.785294,0.01);

\end{tikzpicture}
\end{center}
\end{figure}

\begin{figure}[h!]
\begin{center}
\begin{tikzpicture}[scale=12.5]
\draw (0, 0) -- (1,0);
\draw (0,-0.01) node[below, black]{\scriptsize 0} -- (0,0.01);
\draw (1,-0.01) node[below, black]{\scriptsize 1} -- (1,0.01);
\draw (0.392647,-0.01) node[right, below, black]{\scriptsize$\beta$}-- (0.392647,0.01);
\draw (0.785294,-0.01) node[right, below, black]{\scriptsize$2\beta$}-- (0.785294,0.01);
\draw (0.939465,-0.01) node[below, black]{\scriptsize$2\beta+\beta^2$}-- (0.939465,0.01);
\draw (0.546818,-0.01) node[right, below, black]{\scriptsize$\beta+\beta^2$}-- (0.546818,0.01);
\draw (0.154171,-0.01) node[left, below, black]{\scriptsize$\beta^2$}-- (0.154171,0.01);
\draw (0.308342,-0.01) node[left, below, black]{\scriptsize$2\beta^2$}-- (0.308342,0.01);
\draw (0.700989,-0.01) node[left, below, black]{\scriptsize$\beta+2\beta^2$}-- (0.700989,0.01);
\draw (0.368877,-0.01) -- (0.368877,0.01);
\draw[->] (0.368877,-0.07) node[below, black]{\scriptsize$2\beta^2+\beta^3$}-- (0.368877,-0.02);
\draw (0.761524,-0.01)-- (0.761524,0.01);
\draw[->] (0.761524,-0.07) node[below, black]{\scriptsize$\beta+2\beta^2+\beta^3$}-- (0.761524,-0.02);
\draw (0.060535,-0.01) node[right, below, black]{\scriptsize$\beta^3$}-- (0.060535,0.01);
\end{tikzpicture}
\end{center}
\end{figure}
The sequence of partitions is obtained by splitting at each step the longest interval into $2$ long intervals
followed by an interval of medium length and by  a short interval. So at the first partition we split the 
unit interval into two intervals of length $\beta$, followed by one of length $\beta^2$ and one of length $\beta^3$.
At the second step we split only the first two intervals proportionally into two intervals of length $\beta^2$, one
of lenght $\beta^3$ and one of length $\beta^4$, respectively. The procedure goes on in this way as
for the classical $LS$-sequences.

The associated sequence of points is the sequence of left endpoints of the intervals determining the sequence of 
partitions, where the order is determined by Definition~\ref{seq_of_point}.

Let us list the digit expansion of the first 10 positive integers and the corresponding points of the sequence.

\begin{center}
\begin{tabular}{rlcl}

$\Phi(1)=$
&
$((1,0))$ & $\rightarrow$
&
$\beta$\\
$\Phi(2)=$
&
$((1,1))$ & $\rightarrow$
&
$2\beta$
\\$\Phi(3)=$
&
$((1,2))$ & $\rightarrow$
&
$2\beta+\beta^2$
\\$\Phi(4)=$
&
$((1,0),(0,0))$ & $\rightarrow$
&
$\beta^2$
\\$\Phi(5)=$
&
$((1,0),(1,0))$ & $\rightarrow$
&
$\beta+\beta^2$
\\$\Phi(6)=$
&
$((1,1),(0,0))$ & $\rightarrow$
&
$2\beta^2$
\\$\Phi(7)=$
&
$((1,1),(1,0))$ & $\rightarrow$
&
$\beta+2\beta^2$
\\$\Phi(8)=$
&
$((1,2),(0,0))$ & $\rightarrow$
&
$2\beta^2+\beta^3$\\
$\Phi(9)=$
&
$((1,2),(1,0))$ & $\rightarrow$
&
$\beta+2\beta^2+\beta^3$
\\$\Phi(10)=$
&
$((1,0),(0,0),(0,0))$ & $\rightarrow$&
$\beta^3$
\end{tabular}
\end{center}
\end{example}

\begin{remark}
Let us point out the connection of the generalized $LS$-sequences to the classical $LS$-sequences and van der Corput sequences:
 \begin{description}
  \item[The case $k=1$] In the case $k=1$ we have $\beta=1/L_1$ and the numeration introduced above is the usual $L_1$-adic numeration.
  Indeed, let us note that $l_n=t_n=T_n=L_1^n$ and in particular we obtain
  $$N=\sum_{i=0}^{n} (\epsilon_i T_i+\eta_il_{i,1})=\sum_{i=0}^{n} d_i L_1^n, $$
  where $d_i=\epsilon_i+\eta_i\in\{0,1,\ldots,L_1-1\}$. A close look on Definition \ref{seq_of_point} reveals that in this case the generalized
  $LS$-sequence coincides with the van der Corput sequence.
  \item[The case $k=2$] In this case the generalized $LS$-sequence coincides with the classical $LS$-sequence. Let us note that
  in the case $k=2$ we have that $T_n=l_{n,1}+l_{n,2}=t_n$, i.e.
  $$N=\sum_{i=0}^{n} (\epsilon_i T_i+\eta_il_{i,1})=\sum_{i=0}^{n} (\epsilon_i t_i+\eta_il_{i,1})$$
  which correspondes to the numeration system introduced by Aistleitner et.al. \cite{ahz}. 
 \end{description}
\end{remark}

In order to give an estimate for the discrepancy, it is necessary to introduce the notion of elementary intervals.
An interval is called \emph{elementary} if it is an element of $\rho^n_{L_1,\dots, L_k} \omega$ for some $n$. Equivalently we can define elementary intervals 
as all intervals of the form $I_x^{(m)}=[\xi_{L_1,\dots, L_k}^x,\xi_{L_1,\dots, L_k}^x+\beta^m)$ for some
$m$, where 
$$\Phi(x+l_{m-1,1})=((1,\eta_{m-1}),\dots,(\epsilon_0,\eta_0))$$ with $\eta_{m-1}<L_1$. In particular,
there exists an integer $y<t_m$ such that
$\xi_{L_1,\dots, L_k}^x+\beta^m=\xi_{L_1,\dots, L_k}^y$. Obviously we may choose $y=x+l_{m-1,1}<t_{m}$. 

The next step consists in finding a method to decide
whether a point $\xi_{L_1,\dots, L_k}^N$ is contained in some given elementary interval or
not.

\begin{lemma} \label{Lem:Carbone}
Let $I_x^{(m)}=[\xi_{L_1,\dots, L_k}^x,\xi_{L_1,\dots, L_k}^x+\beta^m)$ be an
elementary interval. Then $\xi_{L_1,\dots, 
L_k}^N \in I_x^{(m)}$ if and only if 
\[x=\sum_{i=0}^{m-1} (\epsilon_iT_i+\eta_il_{i,1})\]
is the truncated representation of $N$.

In addition let $A_x^{(m)}(N)=\sharp \{l\: :\: l\leq N, \xi_{L_1,\dots, L_k}^l \in
I_x^{(m)}\}$ and assume that
\[N=x+\sum_{i=m}^{n}(\epsilon_i T_{i}+\eta_i l_{i,1}).\]
Then 
\[A_x^{(m)}(N)= \sum_{i=0}^{n-m}(\epsilon_{i+m} T_{i}+\eta_{i+m} l_{i,1}) +1.\]
\end{lemma}

\begin{proof}
As soon as we have shown that all integers $N$ of the form 
\[N=x+\sum_{i=m}^{n}(\epsilon_i T_{i}+\eta_i l_{i,1})\]
have the same first $m$ digits as $\Phi(x)$ we can proceed as in the proof of \cite[Lemma~4]{ahz}. But, that all those $N$ have the same first $m$
digits is evident from the constraints on $I_x^{(m)}$ to be an elementary interval, i.e. $x$ is such that 
$$\Phi(x+l_{m-1,1})=((1,\eta_{m-1}),\dots,(\epsilon_0,\eta_0))$$
with $\eta_{m-1}\leq L_1-1$, and Lemma \ref{Lem:Digits}.
\end{proof}

Next we are interested in an accurate formula for $\frac{A_x^{(m)}(N)}N$,
where $\xi_{L_1,\dots, L_k}^N \in I_x^{(m)}$.

\begin{lemma}\label{Lem:Discrepancy-1-dim}
Let us assume that $N$ has a representation of the form \eqref{Rep:N}, and assume that
$\xi_{L_1,\dots, L_k}^N \in I_x^{(m)}$. Then we have
\begin{equation}\label{discr_general}
\frac{A_x^{(m)}(N)}N =\beta^m+\frac{R}{N}\ ,
\end{equation}
where
$$R\leq 1+|\lambda_{1,0}|+\sum_{j=2}^k (2\Lambda_j +|\lambda_{j,0}|).$$
and 
$$\Lambda_j=\max_{\ell=2,\dots,k}\left\{ \frac{|\sum_{i=1}^\ell |\lambda_{j,i}|+(L_1+\dots+L_k-2)|\lambda_{j,1}|}{1-|\beta_j|^{-1}}\right\}.$$
\end{lemma}

Remind that we assume that the polynomial $L_kX^k+\dots+L_1X -1$ has no double roots and has only one positive
root $\beta<1$.

\begin{proof}
 Using our assumptions and Lemma~\ref{Lem:Carbone} we can calculate the exact values of $A_x^{(m)}(N)$ and $N$. In fact, we have
\[N=x + \sum_{i=m}^{n}(\epsilon_i T_{i}+\eta_i l_{i,1})\quad \text{and}
\quad A_x^{(m)}(N)=\sum_{i=m}^{n}(\epsilon_i \tilde T_{i-m}+\eta_i l_{i-m,1}) +1,\]
where $\tilde T_{i-m}=l_{i-m,1}+\dots+l_{i-m,j}$ such that $j$ is the integer such that
$\epsilon_{i+1}=\dots=\epsilon_{i+j}=0$ and $\epsilon_{i+j+1}\neq 0$. In particular, if we write
$$T_i=\sum_{j=1}^k \beta_j^{-i} \Lambda_{j,i},$$
then we have
$$\tilde T_{i-m}=\sum_{j=1}^k \beta_j^{-i+m} \Lambda_{j,i},$$
where $\Lambda_{j,i}=\lambda_{j,1}+\dots +\lambda_{j,\ell_i}$ for some integer $2\leq \ell_i\leq k$ depending on $i$.
This yields
\begin{align*}
 \frac{A_x^{(m)}(N)}N =&\frac{\sum_{i=m}^{n}(\epsilon_i \tilde T_{i-m}+\eta_i l_{i-m,1}) +1}{\sum_{i=m}^{n}(\epsilon_i T_{i}+\eta_i l_{i,1}) +x}\\
=& \frac{\sum_{i=m}^{n}(\epsilon_i \Lambda_{1,i}+\eta_i \lambda_{1,1})\beta^{-i+m}+
\sum_{j=2}^k\tilde R_j+1}
{\beta^{-m}\sum_{i=m}^{n}(\epsilon_i \Lambda_{1,i}+\eta_i \lambda_{1,1})\beta^{-i+m}+
\sum_{j=2}^k \tilde R_j \beta_j^{-m} +x},
\end{align*}
where
$$\tilde R_j=\sum_{i=m}^{n}(\epsilon_i\Lambda_{j,i}+\eta_i \lambda_{j,1})\beta_j^{-i+m}\leq \Lambda_j$$
for $j=2,\dots,k$.
Further, note that
$$N=\beta^{-m}\sum_{i=m}^{n}(\epsilon_i \Lambda_{1,i}+\eta_i \lambda_{1,i})\beta^{-i+m}+\sum_{j=2}^k \tilde R_j \beta_j^{-m} +x $$
Therefore we obtain
\begin{align*}
 \frac{A_x^{(m)}(N)}N=&\beta^m+\frac{\sum_{j=2}^k \tilde R_j+1}N+\frac{\sum_{i=m}^n (\epsilon_i \Lambda_{1,i}+\eta_i \lambda_{1,1})\beta^{-i+m}-N\beta^m}N\\
 =&\beta^m+\frac{\sum_{j=2}^k \tilde R_j+1}N+\frac{\sum_{i=m}^n (\epsilon_i \Lambda_{1,i}+\eta_i \lambda_{1,1})\beta^{-i+m}}N\\
 &-\frac{\sum_{i=m}^{n}(\epsilon_i \Lambda_{1,i}+\eta_i \lambda_{1,1})\beta^{-i+m}+\sum_{j=2}^k \tilde R_j \left(\frac{\beta}{\beta_j}\right)^m+x\beta^m}N\\
=&\beta^m+\frac{1+\sum_{j=2}^k \tilde R_j\left(1-\left(\frac{\beta}{\beta_j}\right)^m\right)-x\beta^m}N.
\end{align*}

Finally we want to estimate $x\beta^m$. Since we assume that $\Psi^{-1}(x+l_{m-1,1})=((1,\eta_{m-1}),\dots,(\epsilon_0,\eta_0))$,
with $\eta_{m-1}\leq L_1-1$ we obviously have that 
$$x<t_{m-1}+l_{m-1,1}(L_1-1)<t_m,$$
and therefore 
$$
x\beta^m\leq  |\lambda_{1,0}|+\sum_{j=2}^k |\lambda_{j,0}|\left|\frac{\beta}{\beta_j}\right|^m
\leq  \sum_{j=1}^k |\lambda_{j,0}|.
$$

If we put all our results together and note that $|1-(\beta/\beta_j)^m|\leq 2$ for each positive integer $m$ we obtain the statement of the Lemma.
\end{proof}

\begin{remark}\label{rem:why_Pisot}
We note that the only place, where we used the fact that the polynomial $L_kX^k+\dots+L_1X-1$ has no double zeros and all roots but one have absolute value
smaller than $1$ is in the proof of Lemma \ref{Lem:Discrepancy-1-dim}. Let us note that dropping the assumption that there exists no double zeros would not 
change the result. But, it would result in a slightly worse estimate for $R$ and more technical difficulties in the course of the proof. For the sake of simplicity we stick with 
the case of simple zeros. 

Further let us note that assuming that all roots but one have absolute value $\leq 1$ would result in an estimate of the form
$$|R|\leq (\log N)^A C_{L_1,\dots,L_k}$$
where $A$ is the number of roots with absolute value $=1$ counted with multiplicities and $C_{L_1,\dots,L_k}$ is some constant depending on the $k$-tuple 
$(L_1,\dots,L_k)$. Plugging this bound into the proof of Theorem \ref{Th:Discrepancy} we would obtain
$$D_N(\xi_{L_1,\dots,L_k}^{1}, \xi_{L_1,\dots,L_k}^{2},\ldots, \xi_{L_1,\dots,L_k}^{N})\leq C_{L_1,\dots,L_K}\frac{(\log N)^{A+1}}N $$

In case that $L_kX^k+\dots+L_1X-1$ has more than one root with absolute value $>1$ the bound of the discrepancy 
would be only of the order $N^{-B}$ for some $B<1$. Let us note that similar observations have been made by Carbone \cite{carbone} in the case
that $k=2$.
\end{remark}


In the next section we compute the discrepancy of classical $LS$-sequences. To do so we need the following Lemma proved in \cite[Lemma 6]{ahz} which is more 
precise than Lemma \ref{Lem:Discrepancy-1-dim}.

\begin{lemma}\label{Lem:Discrepancy-1-dim-classical}
Assume that $N$ has a representation of the form \eqref{Rep:N}, and assume that $\xi_{L,S}^N \in I_x^{(m)}$. Then we have
\begin{equation}\label{discr} \frac{A_x^{(m)}(N)}N =\beta^m+\frac{R(1-(-S\beta)^m)+1-x\beta^m}N,\end{equation}
where
\[R=\sum_{i=m}^{n}(\epsilon_i \tau_1+\eta_i \lambda_1)(-S\beta)^{i-m},\]
with
\begin{equation*}
 \tau_1=\frac{-L-2S+\sqrt{L^2+4S}}{2\sqrt{L^2+4S}}\qquad \lambda_1=\frac{-L+\sqrt{L^2+4S}}{2\sqrt{L^2+4S}}\ .
\end{equation*}

Moreover $R$ can be estimated by

\[|R|<\max\{|\tau_1|,|\tau_1+(L+S-2)\lambda_1|\}\frac{1-(S\beta)^{n-m+1}}{1-S\beta}.\]
if $S\beta\neq 1$ and 
\[|R|<\max\{|\tau_1|,|\tau_1+(L+S-2)\lambda_1|\}\max\{n-m+1,0\}\]
if $S\beta=1$.
\end{lemma}

\section{Discrepancy bounds for classical $LS$-sequences}\label{Sec:ProofTh1}

This section is devoted to the explicit computation of the discrepancy of classical $LS$-sequences with the aim proving 

\begin{theorem}\label{constant}
Let $(\xi_{L,S}^n)_{n\in\mathbb{N}}$ be a classical-$LS$-sequence of points with $L\geq S$. Then
\begin{equation}
D_N(\xi_{L,S}^{1}, \xi_{L,S}^{2},\ldots, \xi_{L,S}^{N})\leq \frac{\log N}{N|\log \beta|} (2L+S-2)\left(\frac{\tilde R}{1-S\beta}+1\right)+\frac BN ,
\end{equation}
with $\tilde{R}=\max\{|\tau_1|,|\tau_1+(L+S-2)\lambda_1|\}$ and
$$B= (2L+S-2)\left(\frac{\tilde R}{1-S\beta}+1\right)+2$$ 
\end{theorem}

\begin{proof}
Let $\xi_{L,S}^{1}, \xi_{L,S}^{2},\ldots, \xi_{L,S}^{N}$ be the first $N$ points of the sequence $(\xi_{L,S}^n)_{n\in\mathbb{N}}$
and take $n=\max\{k :t_k\leq N< t_{k+1}\}$. Consider an arbitrary subinterval $[x,y)$ in $[0,1[$. We want to estimate the number of
points among $\xi_{L,S}^{1}, \xi_{L,S}^{2},\ldots,\xi_{L,S}^{N}$ which lie in the interval. To do so we approximate the interval $[x,y)$
from above by elementary intervals of length at least $\beta^{n+1}$ and apply \eqref{discr} to each interval. 
The points $x$ and $y$ belong to some interval determining the $n$-th partition respectively.
In particular let us assume that $x$ lies in an elementary interval $[x'',x')$ and $y$ lies in an elementary interval $[y',y'')$ of length
$\beta^n$ respectively.

We approximate the interval $[x,y)$ from above, i.e. we try to cover the interval $[x',y')\subseteq [x,y)$ by as few as possible, disjoint elementary
intervals of length at least $\beta^{n+1}$. Starting from the point $x'$ we move to the right in order to reach the point $y'$ with steps of variable length.
By the definition of our $LS$-sequence there is an integer $0\leq \ell\leq L-1$ such that $x'_{1}=x'+\ell \beta^n+S\beta^{n+1}$ is the left endpoint
of an elementary interval of length at least $\beta^{n_1}\geq\beta^{n-1}$. In case that $y'\leq x'_1$ we have found our covering. On the other hand
there exists an integer $0\leq \ell\leq L-1$ such that $y'_1=y'-\ell\beta^n$ is the right endpoint of an elementary interval of length at least
$\beta^{n-1}$. Hence we are left by the problem to cover the interval $[x'_1,y'_1)$ by as few as possible, disjoint elementary
intervals of length at least $\beta^{n}$. It is now easy to see by induction on $n$ that $[x',y')$ is covered by at most $2L+S-2$ disjoint elementary
intervals of length $\beta^2,\dots,\beta^n$ respectively and at most $S$ intervals of length $\beta^{n+1}$ and $L$ invervals of length $\beta$.

Let $A_{x,y}(N)=\sharp \{l\: :\: l\leq N, \xi_{L,S}^l \in [x,y)\}$. Since by construction there are no $\xi_{L,S}^l\in[x',x)\cup[y,y')$
with $l\leq N$ we obtain by Lemma \ref{Lem:Discrepancy-1-dim-classical}
\begin{equation}\label{eq:disc_est_I}
 \begin{split}
\left|(y-x)-\frac{A_{x,y}(N)}N\right|\leq & |x-x'|+|y-y'|+\left|(y'-x')-\frac{A_{x',y'}(N)}N\right|\\
\leq & 2\beta^n+(2L+S-2)\sum_{m=0}^{n+1}\left|\frac{R(1-(-S\beta)^m)+1}N \right|
\end{split}
\end{equation}
where we used the inequality
$$1\geq 1-x\beta^m\geq 1-l_m \beta^m=1-\lambda_0-\lambda_1(-S\beta^2)^m>1-\lambda_0-\lambda_1=0,$$
where $l_m=l_{m,1}=\lambda_0 \beta^m +\lambda_1 (-S\beta)^m$ and $\lambda_0=\frac{L+\sqrt{L^2+4S}}{2\sqrt{L^2+4S}}$.
In order to establish Theorem \ref{constant} we are left to estimate the sum in \eqref{eq:disc_est_I}.
Since $|R|\leq \tilde R \frac{1-(S\beta)^{n-m+1}}{1-(S\beta)}$ we find
\begin{equation*}
\begin{split}
\left|\sum_{m=0}^{n+1}|R(1-(-S\beta)^m)|\right|\leq & \tilde R \sum_{m=0}^{n+1}\frac{(1-(-S\beta)^m)(1-(S\beta)^{n-m+1})}{1-(S\beta)}\\
\leq & \tilde R \sum_{m=0}^{n+1}\frac{(1+(S\beta)^m)(1-(S\beta)^{n-m+1})}{1-(S\beta)}\\
\leq & \frac{(n+2)\tilde R}{1-S\beta}
\end{split}
\end{equation*}
Since $\beta^{-n}\leq t_n\leq N$ we get $n\leq \frac{\log N}{|\log \beta|}$ and we obtain Theorem \ref{constant}.
\end{proof}

As an example we can compute the discrepancy of a particular $LS$-sequence obtained by taking $L=S=1$. This sequence, also called Kakutani-Fibonacci sequence, has been also analyzed in detail in 
\cite{civ, hit} in the frame of ergodic theory where it has been shown that it can be obtained as the orbit of an ergodic transformation.

Take $L=S=1$, then by Theorem \ref{constant} we have that
\begin{equation*}
 D_N(\xi_{1,1}^{1}, \xi_{1,1}^{2},\ldots, \xi_{1,1}^{N})\leq 2.366\frac{\log N}{N}+\frac{3.139}N
\end{equation*}
In the case that $L=10$ and $S=1$ we obtain  
\begin{equation*}
 D_N(\xi_{10,1}^{1}, \xi_{10,1}^{2},\ldots, \xi_{10,1}^{N})\leq 8.66\frac{\log N}{N}+\frac{22.02}N
\end{equation*}

In particular we obtain

\begin{corollary}
Let $S$ be fixed and assume that $L$ is large. Then we obtain
  $$\lim_{N \rightarrow \infty} \frac{ND_N(\xi_{L,S}^{1},\ldots, \xi_{L,S}^{N})}{\log N}\sim \frac{2L}{\log L}$$
as $L\rightarrow \infty$.
\end{corollary}

\begin{proof}
By the formula given for $\tilde R$ in Lemma \ref{Lem:Discrepancy-1-dim-classical} and the fact that $\beta=\frac{-L+\sqrt{L^2+4S}}{2S}$ we obtain
that $\tilde R\sim S\beta \sim S/L$ hence
$$\frac{2L+S}{|\log \beta|}\left(\frac{\tilde R}{1-S\beta}+1\right)\sim \frac{2L}{\log L}$$
as $L\rightarrow \infty$ and $S$ is fixed.
\end{proof}

\begin{remark}\label{rem:1/2}
 One can rather easily improve our bound for the discrepancy by a factor $1/2$. This can be done by balancing our choice of the intervals that
 cover $[x,y)$. Let us assume that $x'$ is the right endpoint of an interval of length $\beta^n$ and let us assume that $x'$ lies nearer to a right
 endpoint $x_1'$ than to the left end point and we go from $x'$ to the right instead to the left and proceed in this manner,
 we need fewer intervals (roughly $1/2$-times fewer) to cover the interval $[x',y')$. We did not work out the details for this improvement,
 since the paper is already rather technical and this approach would further increase the technical difficulties.
\end{remark}

\begin{remark}
Faure \cite{Faure:1981} (see also \cite[page 25]{Niederreiter:RandomNumber}) obtains the upper bound 
$$\lim_{N\rightarrow \infty} N D_N(\xi^1_L,\dots,\xi^N_L)\sim \frac{L}{4\log L}$$
if $L$ is large, where $\xi_L$ denotes
the van der Corput sequence for base $L$. This shows that our
approach gives up to a factor $8$ (respectively $4$ considering the remark above) a similar main term for the discrepancy as one obtains for the
van der Corput sequence. 
\end{remark}

At the end of this section we want to compare our approach to obtain explicit bounds for the discrepancy of $LS$-sequences to the approach
due to Carbone \cite{carbone}.
Therfore we give explicit bounds for the star-discrepancy of the Kakutani-Fibonacci sequence of partitions $D_n(\rho_{1,1}^n)$ and of points
$D_N^*(\xi_{1,1}^n)$, where
$$D_n(\rho_{1,1}^n)=\sup_{0\leq b \leq 1} \left|\frac{1}{t_n}\sum_{i=1}^{t_n} \mathbf{1}_{[0, b)}(t_i^n)-b \right|.$$

Following \cite{carbone}, in order to find upper and lower bounds for $D_n(\rho_{1,1}^n)$, we need to estimate
\begin{equation*}
\frac{1}{t_n}\sum_{i=1}^{t_n} \mathbf{1}_{[0, b)}(t_i^n)-b ,
\end{equation*}
Thus we consider $[0, b)$ as a union of intervals defining the $p$-th
partitions $\rho_{1,1}^p$, for $p\leq n-1$. Therfore let us count how many consecutive intervals
$I_1^1,I_2^1,\dots I_{m_1}^1$ of $\rho_{1,1}^1$ are contained in $[0, b)$. Now, we count how many consecutive intervals $I_1^2,I_2^2,\dots I_{m_2}^2$
of $\rho_{1,1}^2$ are contained in $[0, b)\setminus
\bigcup_{i=1}^{m_1}I_i^1$ and so on. Of course it may happen that $m_k=0$ for some
$k$. Going on with this procedure we get
\begin{equation*}
\bigcup_{i=1}^{m_{n-1}}I_i^{n-1}\subset [0, b)\setminus
\bigcup_{p=1}^{n-2}\left(\bigcup_{i=1}^{m_p}I_i^p\right)\ .
\end{equation*}
Thus
\begin{equation*}
[0, b )=\bigcup_{p=1}^{n-1}\left(\bigcup_{i=1}^{m_p}I_i^p\right)\ .
\end{equation*}
In particular, we need to compute $m_p$. Since $L=S=1$, we have that if
$I_1^p=L^p=I_2^p$, then $I_1^p\cup I_2^p=I_i^{p-1}$ for some $i$, and if
$I_1^p=I_2^p=L_p$, then $I_1^p=S^{p-1}$. In particular, this shows that $m_p=1$.
Thus there are no consecutive intervals of the same partition. After some tedious computations, as in \cite[Equation 10]{carbone}, we finally get 
\begin{align*}
& \frac{1}{t_n}\sum_{i=1}^{t_n}
\mathbf{1}_{[0, b)}(t_i^n)-b=\\
&=\frac{1}{t_n}\frac{(\beta-1)}{(1+\beta^2)}(\beta^2(1-(-\beta)^n)-\beta+(-\beta)^n\beta^n(1+\beta)-(-\beta)^n)\ .
\end{align*}
By considering separately the case that $n$ is even and the case that $n$ is odd, one gets the following bounds:
\begin{equation*}
\frac{0.0652}{t_n}\leq \frac{1}{t_n}\sum_{i=1}^{t_n} \mathbf{1}_{[0, b)}(t_i^n)-b\leq \frac{0.4068}{t_n} \qquad {\rm if}\ n\ {\rm even}
\end{equation*}
and
\begin{equation*}
\frac{-0.2764}{t_n}\leq\frac{1}{t_n}\sum_{i=1}^{t_n} \mathbf{1}_{[0, b)}(t_i^n)-b \leq \frac{0.2764}{t_n} \qquad {\rm if}\ n\ {\rm odd}\ .
\end{equation*}
Thus
\begin{equation*}
D_n(\rho_{1,1}^n)\leq \frac{c_2}{t_n}
\end{equation*}
with $c_2=0.4068$.

Now, let $(\tilde \rho_{L,S}^n)_{n\in\mathbb{N}}$ be the sequence of long intervals $l_n$ of $\rho_{L,S}^n$. 
Then it is possible to bound its discrepancy (see \cite[Proposition 3.4]{carbone}) and obtain for $S<L+1$ and every $n\in\mathbb{N}$
\begin{equation}
\frac{\tilde c_1}{l_n}\leq D_n(\tilde \rho_{L,S}^n)\leq \frac{\tilde c_2}{l_n}\ ,
\end{equation} 
where $\tilde c_1$ and $\tilde c_2$ are constants independent of $n$. Furthermore, it can be shown that if $S<L+1$, then
\begin{equation*}
D_N^*(\xi_{L,S}^1,\dots,\xi_{L,S}^N)\leq 2\frac{c_2+(L+S-2)\tilde c_2}{N|\log \beta|}\log N+(L+S-2)\ ,
\end{equation*}
where $c_2$ is the constant in the upper bound of $D_n(\rho_{L,S}^n)$ (see \cite[Equation 29]{carbone}). 

Finally, plugging in $L=S=1$, we can give an estimate for the discrepancy of the sequence $(\xi_{1,1}^n)_{n\in\mathbb{N}}$.
\begin{equation*}
D_N^*(\xi_{1,1}^N)\leq 2\frac{c_2}{|\log\beta|}\frac{\log N}{N}=2\frac{0.4068}{|\log\beta|}\frac{\log N}{N}\sim 1.6911\frac{\log N}{N}\ .
\end{equation*}
Let us point out that this result is in accordance with the following relation
\begin{equation*}
D_N^*(x_n)\leq D_N(x_n)\leq 2D_N^*(x_n)\ ,
\end{equation*}
which holds true for every sequence of points $(x_n)_{n\in\mathbb{N}}$.

\section{The Discrepancy of generalized $LS$-sequences}

The aim of the present section is to provide bounds for the generalized-$LS$-sequences:

\begin{theorem}\label{Th:Discrepancy}
Let $(\xi_{L_1,\dots,L_k}^n)_{n\in\mathbb{N}}$ be a generalized-$LS$-sequence of points. Then
\begin{equation}
\begin{split}
D_N(\xi_{L_1,\dots,L_k}^{1}, & \ldots, \xi_{L_1,\dots,L_k}^{N})\leq\\ & \frac{\log (N+1)-\log|\lambda_{1,0}\beta^k|}{N|\log \beta|}
(2L_1+L_2+\dots+L_k-2)\tilde R,
\end{split}
\end{equation}
where
$$\tilde R=1+|\lambda_{1,0}|+\sum_{j=2}^k (2\Lambda_j +|\lambda_{j,0}|)$$
and 
$$\Lambda_j=\max_{\ell=2,\dots,k}\left\{ \frac{|\sum_{i=1}^\ell |\lambda_{j,i}|+(L_1+\dots+L_k-2)|\lambda_{j,1}|}{1-|\beta_j|^{-1}}\right\}$$
provided $N$ is large enough.
\end{theorem}

\begin{proof}
The proof of Theorem \ref{Th:Discrepancy} runs along the same lines as the proof of Theorem~\ref{constant}. 

As above we put $n=\max\{m :t_m\leq N< t_{m+1}\}$ and
we consider an arbitrary subinterval $[x,y)\subset [0,1)$ and approximate it
from above by elementary intervals. As in the $LS$-case we assume that $x\in [x'',x')$ and $y \in [y',y'')$ lie in elementary intervals
of length $\beta^n$ respectively.

Let $x_1'$ be the next left endpoint of an elementary interval of length at least $\beta^{n-1}$ from $x'$. We need at most $L_1+\dots+L_k-1$
intervals of variable length to cover the interval $[x',x_1')$. Similarly we proceed for the interval $[y_1',y')$, where $y_1'$ is the nearest
right endpoint of an elementary interval of length at least $\beta^{n-1}$ lying left of $y'$. Obviously we need at most $L_1-1$ elementary intervals
of length $\beta^{n}$ to cover $[y_1',y')$. Similar as in the proof of Theorem \ref{constant} we conclude that we need at most
$2L_1+L_2+\dots+L_k-2$ elementary intervals of each length $\beta^{\ell}$ with $\ell=1,\dots,n+k$ to cover $[x',y')$.

Using similar notations as in the proof of Theorem \ref{constant} and by Lemma \ref{Lem:Discrepancy-1-dim} we obtain

\begin{equation}\label{eq:disc_est_II}
 \begin{split}
\left|(y-x)-\frac{A_{x,y}(N)}N\right|\leq & |x-x'|+|y-y'|+\left|(y'-x')-\frac{A_{x',y'}(N)}N\right|\\
\leq & 2\beta^n+(2L_1+\dots+L_k-2)(n+k)\tilde R,
\end{split}
\end{equation}
where
$$\tilde R=1+|\lambda_{1,0}|+\sum_{j=2}^k (2\Lambda_j +|\lambda_{j,0}|)$$
and 
$$\Lambda_j=\max_{\ell=2,\dots,k}\left\{ \frac{|\sum_{i=1}^\ell |\lambda_{j,i}|+(L_1+\dots+L_k-2)|\lambda_{j,1}|}{1-|\beta_j|^{-1}}\right\}.$$

Since $|\lambda_{1,0}|\beta^{-n}-1\leq t_n\leq N$, provided $N$ is large enough, we get 
$$n\leq \frac{\log (N+1)-\log|\lambda_{1,0}|}{|\log \beta|}$$
and we obtain Theorem \ref{Th:Discrepancy}.
\end{proof}

Let us compute the discrepancy for a concrete example:

\begin{example}
Let us consider the case of the $LMS$-sequence considered above (see Example \ref{ex:LMS}), with $L=2$ and $M=S=1$. To ease the
notations we write $l_n=l_{n,1},m_n=l_{n,2}$ and $s_n=l_{n,3}$. First we note that the roots or the polynomial $X^3+X^2+2X-1$ are
$\beta\simeq 0.393$, $\beta_{2,3}\simeq -0.696\pm i1.436$. Moreover, we can apply Cramer's rule as in \eqref{eq:lambdas} to compute the $\lambda_{j,i}$'s, for $1\leq j\leq 3$ and $0\leq i\leq 3$. Then by Theorem \ref{Th:Discrepancy} we have that
\begin{equation*}
 D_N(\xi_{2,1,1}^{1}, \xi_{2,1,1}^{2},\ldots, \xi_{2,1,1}^{N})\leq 51.4562\frac{\log (N+1)}{N}+\frac{122.5173}N\ .
\end{equation*}
\end{example}

Finally let us state the following remark.

\begin{remark}
 Let us note that the proof of Theorem \ref{Th:Discrepancy} leaves a lot of room for improvement. First, as already explained in
 Remark \ref{rem:1/2} by more carefully choosing the intervals which cover $[x',y')$ we might replace the factor $2L_1+L_2+\dots+L_k-2$
 by something like $L_1+L_2+\dots+L_k$. Moreover, the estimates for $\Lambda_j$ and $R$ in Lemma \ref{Lem:Discrepancy-1-dim} are rather
 rough and can be certainly improved in concrete cases. 
\end{remark}


\end{document}